\documentclass[12pt,oneside]{amsart} 
\usepackage{amssymb,amscd}
\usepackage{euscript}

      \theoremstyle{plain}
      \newtheorem{theorem}{Theorem}[section]
      \newtheorem{lemma}[theorem]{Lemma}
      \newtheorem{corollary}[theorem]{Corollary}
      \newtheorem{proposition}[theorem]{Proposition}
      \newtheorem{remark}[theorem]{Remark}
      
      \newtheorem{definition}[theorem]{Definition}        
          
\numberwithin{equation}{section}

      \makeatletter
      \def\@setcopyright{}
      \def\serieslogo@{}
      \makeatother

\def\A{\EuScript{A}} 
\def\B{\EuScript{B}} 
\def\E{\mathcal{V}}

\def\H{\mathcal{H}}
\def\M{\mathcal{M}}
\def\m{m}

\def\R{\mathbb R}

\def\Z{\mathbb Z}
\def\N{\mathbb N}
\def\T{\mathbb T}

\def\dist{\text{dist}}

\def\Id{\text{Id}}
\def\e{\epsilon}
\def\a{\alpha}

\def\b{\beta}

\def\mh{\hat \mu}
\def\nh{\hat \nu}
\def\g{\gamma}

\def\bv{\mathbf v}

\begin{document}

\author{Boris Kalinin$^{\ast}$ and Victoria Sadovskaya$^{\ast \ast}$}

\address{Department of Mathematics, The Pennsylvania State University, University Park, PA 16802, USA.}
\email{kalinin@psu.edu, sadovskaya@psu.edu}

\title [Holonomies and cohomology for cocycles $\;\;$]
{Holonomies and cohomology for cocycles over partially hyperbolic diffeomorphisms} 

\thanks{$^{\ast}$ Supported in part by NSF grant DMS-1101150}
\thanks{$^{\ast \ast}$ Supported in part by NSF grant DMS-1301693}


\begin{abstract} 
We consider  group-valued cocycles over a partially hyperbolic diffeomorphism 
which is accessible volume-preserving and center bunched. We study  cocycles with values in the group of invertible continuous linear operators on 
a Banach space. We describe properties of holonomies for fiber bunched  
cocycles and establish their H\"older regularity.
We also study cohomology of cocycles and its connection with
 holonomies. We obtain a result on regularity of a measurable 
 conjugacy, as well as  a necessary and sufficient condition for 
 existence of a continuous conjugacy  between two cocycles.
\end{abstract}

\maketitle 

\section{Introduction}
 
Cocycles and their cohomology  play an important role in dynamics.
For example, they appear in the study of time changes  
for flows and group actions, existence and smoothness of absolutely 
continuous invariant measures, existence and smoothness of 
conjugacies between dynamical systems, rigidity in dynamical systems
and  group actions.
In this paper we consider cohomology of group-valued cocycles over 
partially hyperbolic diffeomorphisms. 
 
\begin{definition} Let $f$ be a diffeomorphism of a compact manifold $\M$, 
let $G$ be a topological  group 
equipped with a  complete metric, and let 
$A:\M\to G$  be a continuous function. 
The $G$-valued cocycle 
over $f$ generated by $A$ is the map 
 $$\A:\,\M \times \Z \,\to G   \qquad \text{defined  by }$$
 $$
\begin{aligned}
&\A(x,0)=\A_x^0=e_G,  \;\;\;\; \A(x,n)=\A_x^n = 
A(f^{n-1} x)\circ \cdots \circ A(x) \quad \text{ and }\;\; \\
&\A(x,-n)=\A_x^{-n}= (\A_{f^{-n} x}^n)^{-1} = 
 (A(f^{-n} x))^{-1}\circ \cdots \circ (A(f^{-1}x))^{-1}, \quad n\in \N.
\end{aligned}
$$
\end{definition}

 If the tangent bundle of $\M$ is trivial, $T\M= \M\times \R^d$, 
then the differential $Df$ can be viewed as a $GL(d,\R)$-valued cocycle:
$  A(x)=Df_x$ and  $\A_x^n=Df^n_x.$ More generally, one can consider 
the restriction of $Df$ to a continuous invariant sub-bundle of $\, T\M$, 
for example stable, unstable, or center.  In this paper we consider a more general setting of cocycles with values 
in the group of invertible operators on a Banach space.

A natural  equivalence relation for cocycles is defined as follows.  

\begin{definition} \label{cohomology def}
Cocycles $\A$ and $\B$ are 
(measurably, continuously)  {\em cohomologous} 
if there exists a (measurable, continuous) 
function  $C:\M\to G$ such that
$$ \,\A_x^n=C(f^n x) \circ \B_x^n \circ  C(x)^{-1} 
  \;\text{ for all }n\in \Z \text{ and }x\in \M,
  $$ 
equivalently, for the generators  $A(x)=C(fx) \circ B(x) \circ C(x)^{-1}\;\text{ for all }x\in \M.$
\end{definition}

\noindent We refer to $C$ as a {\em conjugacy} between $\A$  and $\B$.
It is also called a {\em transfer map}. For the differential example above, 
$C(x)$ can be viewed as a coordinate change on $T_x\M$.

In the context of cocycles over partially hyperbolic systems, two main
cohomology problems have been considered so far. One is finding 
sufficient conditions for existence of a continuous conjugacy.
The other is determining whether a measurable conjugacy 
between two cocycles is necessarily continuous or more regular. 

For H\"older continuous real-valued cocycles over systems with local accessibility,  the first problem was resolved in  \cite{KK}, where conditions 
for existence of a conjugacy were established in terms of $su$-cycle 
functionals. Recently, the study of real-valued cocycles  was advanced by
A. Wilkinson  in \cite{W}, where she weakened the assumption from  local accessibility to accessibility  and obtained a positive solution for the  second problem. Previous results in this direction were established in \cite{D} 
for smooth real-valued cocycles over systems with rapid mixing. 

For cocycles with values in non-commutative groups, studying cohomology
is more difficult. In all results so far, the cocycles satisfied additional assumptions related to their growth, for example fiber bunching for linear cocycles. 
This property means that noncoformality of the cocycle is dominated
by the contraction/expansion of $f$ in the stable/unstable directions.
Also, some conclusions in the non-commutative case are different from those 
in the commutative case. For example, a measurable conjugacy  between 
two cocycles is not necessarily continuous, even when both cocycles are fiber 
bunched \cite{PW}. Theorem \ref{measurable cohomology} gives the first 
result on continuity of a measurable conjugacy for non-commutative 
cocycles over partially hyperbolic systems. We make an additional
assumption that one of the cocycles  is uniformly quasiconformal.
The assumption is close to optimal and the theorem extends all 
similar results for cocycles over hyperbolic diffeomorphisms 
\cite{Sch,NP,PW,S14}.

We also obtain a necessary and sufficient condition for existence of a 
continuous conjugacy between two cocycles in terms of their $su$-cycle weights. 
Previously, a sufficient condition was obtained in \cite{KN} for conjugacy to a constant cocycle over a system with local accessibility. However, for non-commutative cocycles the general problem cannot be reduced to the case 
when one cocycle is constant. We note that in all our results partial 
hyperbolicity of the base system is  pointwise and accessibility is not 
assumed to be local. The fiber bunching for cocycles is assumed
in pointwise sense so, in particular, the results apply to the derivative 
cocycle along the center direction of a strongly center bunched partially hyperbolic diffeomorphism.

Fiber bunching of a cocycle implies existence of so called stable and 
unstable holonomies. Some of our results make a weaker assumption of
existence of holonomies in place of fiber bunching. Holonomies are an important and convenient tool in the study of cocycles. 
In Theorem \ref{alpha Holder} we 
establish H\"older continuity of holonomies, which is a result 
of independent interest. We also obtain results on  the relationship between 
conjugacy and  holonomies of  cocycles, which turns out
to be more complicated then in the commutative case.  For example, 
$su$-cycle weights may be non-trivial for a cocycle continuously 
cohomologous to a constant one.

In Section 2 we give definitions of  partially hyperbolic diffeomorphisms and Banach cocycles. In Section 3 we discuss holonomies and state 
our result on their regularity. In Section 4 we formulate our results on cohomology of cocycles, and in the last section we give proofs of
all the results.

\section{Preliminaries}

 \subsection{Partially hyperbolic diffeomorphisms.} \label{partial}
 (See  \cite{BW} for more details.)
 
Let $\M$ be a compact connected smooth manifold.
 A diffeomorphism $f$ of $ \M$ 
is said to be {\em partially hyperbolic} if
there exist a nontrivial $Df$-invariant splitting of the tangent bundle 
$T\M =E^s\oplus E^c \oplus E^u,$ and  a Riemannian 
metric on $\M$ for which one can choose continuous positive 
functions $\nu<1,\,$ $\hat\nu<1,\,$ $\gamma,$ $\hat\gamma\,$ such that 
for any $x \in \M$ and unit vectors  
$\,\bv^s\in E^s(x)$, $\,\bv^c\in E^c(x)$, and $\,\bv^u\in E^u(x)$
\begin{equation}\label{partial def}
\|Df_x(\bv^s)\| < \nu(x) <\gamma(x) <\|Df_x(\bv^c)\| < \hat\gamma(x)^{-1} <
\hat\nu(x)^{-1} <\|Df_x(\bv^u)\|.
\end{equation}
We also choose continuous  functions $\mu$ and $\hat \mu$ 
such that for all $x$ in $\M$
\begin{equation}\label{kappa}
\mu(x)<\|Df_x(\bv^s)\|  \,\text{ if }\, \bv^x \in E^s(x)
\quad\text{and}\quad
\|Df_x(\bv^u)\| < \hat\mu (x)^{-1}\text{ if }\, \bv^u \in E^u(x).
\end{equation}
 The sub-bundles $E^s$, $E^u$, and $E^c$ are called, respectively,
 stable, unstable, and center.
$E^s$ and $E^u$  are 
tangent to the stable and unstable foliations $W^s$ and $W^u$
 respectively.  
 
  An {\it $su$-path} in $\M$ is a concatenation 
 of finitely many subpaths which lie entirely in a single 
 leaf of $W^s$ or  $W^u$. A partially hyperbolic  
 diffeomorphism $f$  is called {\em accessible}  if any two points 
 in $\M$ can be connected by an $su$-path.

 We say that $f$ is {\em volume-preserving} if it has an invariant probability 
measure  $\m$ in the measure class of a volume induced by a 
Riemannian metric. It is conjectured that any essentially accessible 
$f$ is ergodic with respect to such  $\m$. The conjecture was proved
in cite \cite{BW} under the assumption that $f$ is $C^2$ and center bunched,
or that $f$ is $C^{1+\epsilon}$, $0<\epsilon<1$, and strongly center bunched.
The diffeomorphism $f$ is called {\em center bunched}\,
if the functions  $\nu, \hat\nu, \gamma, \hat\gamma$  can be 
chosen to satisfy
\begin{equation}\label{center bunching}
\nu<\gamma \hat \gamma \quad \text{and}
\quad \hat\nu<\gamma \hat \gamma .
\end{equation}
A $C^{1+\epsilon}$ diffeomorphism $f$ is called
{\em strongly center bunched}\, if
 \begin{equation}\label{strong center bunching}
\nu^\theta<\gamma \hat \gamma \quad \text{and}
\quad \hat\nu^\theta<\gamma \hat \gamma
\end{equation}
 for some $\theta \in (0, \epsilon)$ satisfying the inequalities
$\,\nu\gamma^{-1}< \mu^\theta$ and
$\,\hat\nu \hat\gamma^{-1}< \hat \mu^\theta$.
These inequalities imply that $E^c$ is $\theta$-H\"older.
Note that $(\gamma \hat \gamma)^{-1}$ is an estimate of 
non-conformality of $Df|_{E^c}$.

 \subsection{Banach  cocycles.} \label{Banach}
Let $V$ be a Banach space, i.e. a vector space equipped with a norm $\|.\|$
such that $V$ is complete with respect to the induced metric. We denote by
$L(V)$ the space of continuous linear operators from $V$ to itself. Then
$L(V)$ becomes a Banach space when equipped with the operator norm 
$$
\|A\|=\sup \,\{ \|Av\| : \,v\in V , \;\|v\| \le 1\}, \quad A \in L(V).
$$ 
We denote by $GL(V)$ the set  of invertible elements in  $L(V)$.
The set $GL(V)$ is an open subset of $L(V)$ and a group with respect
to composition. We use the following metric on $GL(V)$, with respect to which
it is complete,
\begin{equation}  \label{distGLV} 
d (A, B)=\dist_{GL(V)} (A, B) = \| A  - B \|  + \| A^{-1}  - B^{-1} \|.
\end{equation}
We call a $GL(V)$-valued cocycle $\A$ a {\em Banach cocycle}. 
It is called {\em $\beta$-H\"older} if its generator $A : \M \to GL(V)$ is H\"older continuous
with exponent $\beta$ with respect to the metric $d$.
We note that on any compact set $S\subset GL(V)$  the distance $d(A,B)$
is Lipschitz equivalent to $\| A  - B \|$ by Lemma \ref{eqdist} below. 
Therefore, 
since $\M$ is compact, a cocycle $\A$ is $\beta$-H\"older if and only if 
$$
\| A (x)  - A(y) \| \le c \, \dist (x,y)^\beta \quad\text{for all }x,y \in \M. 
$$

\begin{lemma}  \label{eqdist}
Suppose that for a subset $S \in GL(V)$ there exists $M$ such that $\|A\| \le M$
and $\| A^{-1}\| \le M$ for all $A \in S$. Then for all $A,B \in S$ we have
$$
\begin{aligned}
M^{-1}\, \| A^{-1} B-\Id \, \| \,&\le \,\| A  - B \| \,\le \, 
d (A, B) \,=\, d(A^{-1} , B^{-1}) \,\le \\
&\le (M^2+1) \, \| A  - B \| \,\le\, M(M^2+1)\, \| A^{-1} B-\Id \, \|.
\end{aligned}
$$
\end{lemma}
\begin{proof} The equality is clear from the definition of $d$,
and the next inequality follows from the  estimate
$$
\| A^{-1}  - B^{-1} \| \le \|A^{-1} \| \cdot \| B-A \| \cdot \| B^{-1}\| \le M^2 \| A  - B \|.
$$
The other inequalities are obtained similarly.
\end{proof}

\begin{definition} \label{fiber bunched def}
 A cocycle $\A$ over  $f$ is called\, {\em $\beta$ fiber bunched} if it is $\beta$-H\"older and
\begin{equation}\label{fiber bunched}
\| \A(x)\|\cdot \| \A(x)^{-1}\| \cdot  \nu(x)^\beta <1 \quad\text{and}\quad
\| \A(x)\|\cdot \| \A(x)^{-1}\| \cdot  \hat\nu(x)^\beta <1,
\end{equation}
for all $x$ in $\M$, where $\nu$ and $\hat\nu$ are as in \eqref{partial def}.
\end{definition}

This means that nonconformality of $\A$ is dominated by the 
expansion/contraction along unstable/stable foliations in the base.
Note that the cocycle $Df|_{E^c}$ for a strongly center bunched
 \eqref{strong center bunching} partially hyperbolic diffeomorphism
  is $\theta$ fiber bunched.

We can view the generator $A$ as the automorphism of the trivial vector 
bundle $\E=\M \times V$ given by $A(x, v)= (fx, A(x) v)$, and $\A _x ^n$ as a linear
map between the fibers $\E_x$ and $\E_{f^nx}$.  We deal with the case of a trivial 
bundle for convenience. Our results extend directly to linear cocycles defined
more generally as bundle automorphisms, see \cite{KS13} for a description
of this setting.

\subsection{Standing assumptions} 
{\em In  this paper,\\
$\bullet$ $\M$ is a compact connected smooth manifold; 
\vskip.05cm
\noindent $\bullet$  $f$ is an accessible partially hyperbolic diffeomorphism of $\M$
that preserves a volume 

$\m$ and is  
either $C^2$ and center bunched, or $C^{1+\epsilon}$ and strongly center bunched;
\vskip.05cm
\noindent $\bullet$  $\A$ and $ \B$ are $GL(V)$-valued continuous cocycles over $f$, 
where $V$ is a Banach space.}


\section{Holonomies and their regularity}

An important role in the study of cocycles is played by holonomies. They were
introduced by M. Viana in \cite{V} for linear cocycles and further developed 
and used in \cite{ASV,KS13}. For a fiber bunched linear cocycle $\A$, a
holonomy can be obtained as a limit of the products 
$(\A^n_y)^{-1} \circ \A^n_x$.
Convergence and limits of such products have been studied for various types of group-valued cocycles whose growth is slower than the contraction/expansion
 in the base (see e.g. \cite{NT,PW, LW}). It is related to existence of strong stable/unstable manifolds for the extended system on the bundle.  We use the
axiomatic definition of holonomies given in \cite{V,ASV}. We note, however, 
that the resulting object is non-unique in general, see discussion after
Corollary  \ref{cohomologous to constant}.

\begin{definition} \label{holonomies}
A {\em stable holonomy} for a  cocycle $\A$ is 
a continuous map \\ $H^{\A,s}:(x,y)\mapsto H^{\A,s}_{x,y}$, 
where $x\in \M$ and $y\in W^s(x)$, such that  
\begin{itemize}
\item[(H1)] $H^{\A,s}_{x,y}$ is an element of $GL(V)$, viewed as a  map from $\E_x$ to $\E_y$;
\item[(H2)]  $H^{\A,s}_{x,x}=\Id\,$ for every $x\in \M$ and 
$\,H^{\A,s}_{y,z} \circ H^{\A,s}_{x,y}= H^{\A,s}_{x,z}$; 
\item[(H3)]  $H^{\A,s}_{x,y}= (\A^n_y)^{-1}\circ H^{\A,s}_{f^nx ,f^ny} \circ \A^n_x\;$ 
for all $n\in \N$.
\end{itemize}
\end{definition}
\noindent We say that a stable holonomy is {\em $\beta$-H\"older} 
(along the leaves of $W^s$) if it satisfies the following additional property: for any $R>0$ there exists $K$ such that 
\begin{itemize}
 \item [(H4)] $\;\| H^{\A,s}_{x,y} - \Id \,\| \leq K\,\dist_{W^s}  (x,y)^{\beta} \;$
 for any $x \in \M$ and $y \in W^s_{R}(x)$.
 \end{itemize}
Here $\dist_{W^s}$ denotes the distance along a leaf of the stable foliation $W^s$, and
$W^s_{R}(x)$ denotes the ball in $W^s(x)$  centered at $x$ of radius $R$ in this distance.
By Lemma \ref{eqdist}, the left hand side of (H4) is equivalent 
to the $GL(V)$ distance $d\,(H^{\A,s}_{x,y},\Id)$ on the compact set 
$\{ H^{\A,s}_{x,y} :\; x\in \M, \;y \in \overline{W^s_{R}(x)} \,\}.$
\vskip.1cm

Fiber bunched cocycles have a canonical holonomy. The following result was proved for finite dimensional Banach spaces $V$, but the arguments work 
for the general case without any modifications.

\begin{proposition} [Proposition 4.2 \cite{KS13}, cf. Proposition 3.4 \cite{ASV}]\label{holonomies exist} $\;$

\noindent Suppose that a cocycle $\A$ is $\beta$ fiber bunched.
Then for any $x\in \M$ and $y\in W^s(x)$,
\begin{equation}\label{as limits}
  H^{\A,s}_{x,y} \overset{def}{=}\underset{n\to\infty}{\lim} \,(\A^n_y)^{-1} \circ \A^n_x\,
\end{equation} 
 exists and satisfies (H1,2,3,4).
The stable holonomy for $\A$ satisfying (H4) is unique.
 \end{proposition}

\begin{remark} \label{weaker}
This proposition holds under a slightly weaker fiber bunching assumption
\cite[Proposition 4.4]{S14}:   
there exist $\theta<1$ and $L$ such that for all $x\in\M$, $n\in \N$,
\begin{equation}\label{weak fiber bunched}
\| \A_x^n\|\cdot \|(\A_x^n)^{-1}\| \cdot  (\nu^n_x)^\beta < L\, \theta^n \;\text{ and}\quad
\| \A_x^{-n}\|\cdot \|(\A_x^{-n})^{-1}\| \cdot  (\hat \nu^{-n}_x)^\beta < L\, \theta^n,
\end{equation}
where 
$ \;\nu^n_x$, $\hat \nu^n_x$ are defined as in \eqref{def nu_n}.
In fact, all results in this paper hold under this version of fiber 
bunching assumption.
\end{remark}

\begin{definition} \label{standard holonomy}
A stable holonomy for a cocycle $\A$ satisfying  \eqref{as limits} is 
called {\em standard}.
\end{definition}
By definition, the standard stable holonomy of $\A$ is unique, if it exists. 
By the proposition, the only $\beta$-H\"older  stable holonomy for a $\beta$
 fiber bunched cocycle  is the standard one. However, there are non-standard
stable holonomies of lower regularity even for a constant fiber bunched
cocycle over an Anosov automorphism. 

 \vskip.1cm
 
We use similar definitions for an unstable holonomy $H^{\A,u}$. 
As in Proposition \ref{holonomies exist}, any $\beta$ fiber bunched cocycle  
$\A$ has the standard unstable holonomy obtained as 
$$
H^{\A,u}_{x,\,y}\,=
\underset{n\to\infty}{\lim} \left( (\A^{-n}_y)^{-1} \circ (\A^{-n}_x) \right)
=
\underset{n\to\infty}{\lim} 
\left(\A^n_{f^{-n}y} \circ (\A^{n}_{f^{-n}x})^{-1} \right), \quad y\in W^u(x),
$$
It satisfies  (H1,2,4,) above with $y\in W^u(x)$ and 
\vskip.15cm
(H3$'$)  $H^{\A,u}_{x,\,y}=
(\A^{-n}_y)^{-1} \circ H^{\A,u}_{f^{-n}x ,\,f^{-n}y}  \circ \A^{-n}_x\;$ 
 for all $n\in \N$.

\vskip.2cm

We establish global H\"older continuity of the stable holonomy for fiber bunched cocycles. A similar result holds for the unstable holonomy.

\begin{theorem}\label{alpha Holder}
Suppose that a cocycle $\A$ is $\beta$ fiber bunched. Then  there exists   
$\alpha$,  $0< \alpha < \beta$, such that the standard holonomy $H^{\A,s}$ 
as in \eqref{as limits} is globally $\a$-H\"older in the following sense. 
For any $R>0$ there exist $\delta>0$ and $C>0$ so that 
\vskip.2cm

If $\,y\in W^s_{R}(x)$, $\,y''\in W^s_{R}(x'')$, $\,\dist(x,x'')<\delta$
and $\dist(y, y'')<\delta$, then 
\vskip.1cm
$d(H^s_{x,y}, H^s_{x'',y''}) \le C \max \,\{ \dist(x, x'')^\a, \dist(y, y'')^\a \}$.
\end{theorem}

\noindent The choice of the H\"older exponent $\alpha$ is explicit and is
 described in the beginning of the proof. It depends on the system in the 
 base and on the ``relative degree of non-conformality" of the cocycle $\A$.

\vskip.1cm

In the absence of fiber bunching, natural examples of cocycles with standard 
holonomies are given by small perturbation of a constant $GL(d,\R)$-valued 
cocycle. 

\begin{proposition} \label{constant}
Let $\A$ be a constant $GL(d,\R)$-valued cocycle generated by $A$. 
If $B : \M \to GL(d,\R)$ is H\"older continuous and is sufficiently $C^0$ 
close to $A$, then the cocycle generated by $B$ has H\"older 
continuous standard holonomies.
\end{proposition}


\section{Cohomology of cocycles}

First we consider the question whether a measurable conjugacy between 
two cocycles is continuous.  For  non-commutative cocycles,
the answer is not always positive, even when both cocycles are  fiber bunched.  Indeed, in \cite[Section 9]{PW}, M.~Pollicott and 
C. P. Walkden constructed an example of two smooth $GL(2,\R)$-valued cocycles over an Anosov toral automorphism
that are measurably (with respect to the Lebesgue measure), but not continuously cohomologous.
The cocycles can be made  arbitrarily close to the identity and, in particular, fiber bunched.
We establish  continuity of a measurable conjugacy for  fiber bunched 
cocycles under the assumption that one of them is uniformly quasiconformal. The example above shows that this assumption is close to optimal.  

\begin{definition} \label{quasiconformal def}
A cocycle $\B$ is called {\em uniformly quasiconformal}
 if there exists a number $K(\B)$ such that the quasiconformal distortion 
 satisfies 
\begin{equation}\label{quasiconformal}
K_\B(x,n) \overset{def}{=}\|\B^n_x\| \cdot \|(\B^n_x)^{-1}\|\le K(\B) \quad\text{for all }
x\in \M \text{  and }n\in \Z.
\end{equation}
 If $\,K_\B(x,n)=1$ for all $x$ and $n$,
 the cocycle is said to be \em{conformal}.
 \end{definition}
 
 Clearly, H\"older continuous conformal cocycles are fiber bunched, 
 and so are all sufficiently high iterates of uniformly quasiconformal  
 cocycles.

\begin{theorem} \label{measurable cohomology}
Let $\A$ be a cocycle with  standard holonomy and let $\B$ be a uniformly quasiconformal H\"older cocycle.
Let $\m$ be the invariant volume for  $f$, 
and let $C$ be a $\m$-measurable conjugacy between $\A$ and $\B$.  
If $V$ is finite dimensional then $C$ coincides on a set of full measure with a continuous  conjugacy
that intertwines the standard holonomies of $\A$ and $\B$.
\end{theorem}
When we speak of a {\em holonomy} for a cocycle $\A$ 
we mean a pair of a stable  holonomy and an unstable holonomy, 
$\H^{\A}=\{ \H^{\A,s} ,\H^{\A,u} \}$.
When we say that a conjugacy {\em intertwines $H^{\A}$ and $H^{\B}$} we mean
mean that it intertwines both the stable and the unstable holonomies as in the following definition.
\begin{definition}
Suppose that $H^{\A,s}$ and $H^{\B,s}$ are stable holonomies for 
cocycles $\A$ and $\B$. We say that a conjugacy $C$ between 
$\A$ and $\B$  {\em intertwines 
$H^{\A,s}$ and $H^{\B,s}$} if
\begin{equation}\label{intertwines}
H_{x,y}^{\A,s}=C(y)\circ H_{x,y}^{\B,s}\circ C(x)^{-1}\quad \text{for all }x,y \in \M
\text{ such that }y\in W^s(x).
\end{equation}
\end{definition}

Intertwining the standard holonomies of cocycles is an important property 
of a conjugacy $C$. It is clear from the proof that it implies continuity  of $C$. 
Further, it can be uses to study higher regularity of the conjugacy, see
\cite{NT} for results on non-commutative cocycles over hyperbolic systems and
\cite{W} for real-valued cocycles over accessible partially hyperbolic systems.
In contrast to real-valued cocycles, however, even continuous conjugacy between fiber bunched cocycles does not necessarily intertwine their 
standard holonomies.

\begin{proposition} \label{nonreg example}
For any $0<\beta'<\beta\le1$,
there exist a smooth cocycle $\A$ and a constant cocycle $\B$ over 
an Anosov  automorphism of $\,\T^2$ that are $\beta$ fiber-bunched 
and  conjugate via a $\beta'$-H\"older function $C$, but there is no 
$\beta$-H\"older conjugacy between $\A$ and $\B$ and 
no conjugacy intertwines their standard holonomies.
\end{proposition}

The next proposition gives a general sufficient condition for intertwining.

\begin{proposition} \label{cohomology 2}
Suppose that cocycles $\A$ and $\B$ are $\beta$ fiber bunched.
Then any $\beta$-H\"older conjugacy  $C$ between them  intertwines 
their standard holonomies.
\end{proposition} 
 
 \noindent
It is clear from the proof that it suffices to assume $\beta$-H\"older 
continuity of $C$ along the stable/unstable leaves to obtain intertwining 
of the standard stable/unstable holonomies respectively. Conversely, 
intertwining $\beta$-H\"older holonomies implies $\beta$-H\"older 
continuity of $C$ along  the stable and unstable leaves. Then global 
H\"older continuity of $C$ follows for hyperbolic $f$. For a partially 
hyperbolic $f$, accessibility is not known to imply global 
H\"older continuity of $C$, but a stronger assumption suffices. 
The diffeomorphism $f$ is called {\em locally 
$\alpha$-H\"older accessible}\, if there exists a number $L=L(f)$ such 
that for all sufficiently close $x,y\in \M$ there is an $su$-path
$$P=\{x=x_0, \,x_1, \dots ,\, x_L=y\}\,\quad \text{such that}\quad 
  \dist_{W^i}  (x_{i-1}, x_i)\le C\,\dist (x,y)^\alpha
  $$
for $i=1, \dots, L.$
Here the distance between $x_{i-1}$ and $x_i$ is measured  along the 
corresponding stable or unstable leaf $W^i$. Such accessibility implies
$\alpha \beta$-H\"older continuity of $C$, see \cite[Corollary 3.7]{KS13}.
The usual accessibility implies that an $su$-path can be chosen with $L$ and 
the distances $\dist_{W^i}  (x_{i-1}, x_i)$  uniformly bounded.
If, in addition, the points $x_i$ can be chosen to depend H\"older 
continuously on $x$ and $y$, then Theorem \ref{alpha Holder} can
be used to obtain global H\"older continuity of $C$.

\vskip.1cm
 
Now we consider the problem of  finding 
sufficient conditions for existence of a continuous conjugacy
between two cocycles.
Suppose that $H^{\A,s}$ and $H^{\A,u}$ are stable and unstable 
holonomies for a cocycle $\A$.
Let $P=\{x_0, x_1, \dots , x_{k-1}, x_k\}$ be an  $su$-path in $\M$.
We define the {\em weight} of  $P$  as
$$
\H^{\A, P}_{x_0,x_k} =  H_{x_0, x_{k-1}} \circ \dots \circ H_{x_1, x_2} \circ H_{x_0, x_1}, 
$$
where $H_{x_i, x_{i+1}}=H^{s/u}_{x_i, x_{i+1}}$ if $x_{i+1} \in W^{s/u} (x_i)$.
An   {\em $su$-cycle} is an $su$-path in $\M$ with $x_0=x_k$,  and we refer to the corresponding $\H^{\A,P}_{x_0}$ as the {\em cycle weight}.
In case of real-valued cocycles, $\H^{\A,P}_{x_0}$ is also referred to
as the cycle functional.

\vskip.1cm

The following properties are easy to verify. 

\begin{proposition} \label{prop of inter conj}
Let $H^{\A}$ and $H^{\B}$ be holonomies for cocycles $\A$ and $\B$
and let $C$ be a continuous conjugacy between  $\A$ and $\B$
which intertwines these holonomies. Then

\begin{itemize}
\item[(i)]  $C$ conjugates the cycle weights of these holonomies, i.e.
$$
 \H^{\A,P}_x=C(x)\circ \H^{\B,P}_x \circ C(x)^{-1}
\quad\text{ for every }su\text{-cycle }P=P_x.
$$

\item[(ii)] More generally,  for any $x,y\in \M$
and any $su$-path $P_{x,y}$ from $x$ to $y$,
$$
\hskip1cm \H^{\A,P}_{x,y}=C(y)\circ \H^{\B,P}_{x,y}\circ C(x)^{-1}
\, \text{ and hence }\; 
C(y)=\H^{\A,P}_{x,y} \circ C(x)\circ  (\H^{\B,P}_{x,y})^{-1}.
$$

\item[(iii)]  $C$ is uniquely determined by its value at any point.
 \end{itemize}
\end{proposition}

The next theorem gives a sufficient condition for existence of a continuous 
conjugacy intertwining holonomies. By the previous proposition, this condition 
is also necessary.

\begin{theorem} \label{sufficient}
Let $\A$ and $\B$ be cocycles with holonomies  $H^{\A}$ and $H^{\B}$. Suppose that there exist $x_0\in \M$ and $C_{x_0}\in  GL(V)$  such that 
\vskip.15cm
\begin{itemize}
\item[(i)]
$\H^{\A,P}_{x_0}=C_{x_0}\circ \H^{\B,P}_{x_0} \circ C_{x_0}^{-1} \quad 
\text{for every }su\text{-cycle }P_{x_0} $, and
\vskip.1cm
\item[(ii)]  $\A_{x_0}= C_{f{x_0}}\circ \B_{x_0} \circ C_{x_0}^{-1}$,
\,where 
$C_{f{x_0}}=\H^{\A,P}_{{x_0},f{x_0}} \circ C_{x_0}\circ  (\H^{\B,P}_{{x_0},f{x_0}})^{-1}$
for some $su$-path $P_{{x_0},f{x_0}}$ from $x_0$ to $fx_0$.
\end{itemize}
Then there exists a continuous conjugacy $C$ between $\A$ and $\B$ 
with $C({x_0})=C_{x_0}$ that intertwines $H^{\A}$ and $H^{\B}$.
\end{theorem}

We note that due to the first assumption, $C_{f{x_0}}$ in (ii) does not depend
on the choice of a path $P_{x_0,fx_0}$. If $x_0$ is a fixed point for $f$ then, considering the trivial path from 
$x_0$ to $fx_0=x_0$, we see that   condition (ii)  becomes 
$\A_{x_0}=C_{x_0} \circ \B_{x_0} \circ C_{x_0}^{-1}$, and we obtain the following corollary.
Thus, in this case (i) can be viewed 
as a sufficient condition for extending a conjugacy from a given value at a fixed point.

\begin{corollary} \label{fixedpoint}
Let $\A$ and $\B$ be cocycles with holonomies  $H^{\A}$ and $H^{\B}$. 
Suppose that  there exist a fixed point $x_0$ and $C_{x_0}\in  GL(V)$ such that 
 $\A_{x_0}= C_{x_0}\circ \B_{x_0} \circ C_{x_0}^{-1}$ and 
 $\H^{\A,P}_{x_0}=C_{x_0}\circ \H^{\B,P}_{x_0} \circ C_{x_0}^{-1}$
for every $su$-cycle $P_{x_0}$. Then there exists a continuous conjugacy 
$C$ between $\A$ and $\B$ with $C({x_0})=C_{x_0}$ that intertwines 
$H^{\A}$ and $H^{\B}$.
\end{corollary}

Now we apply Theorem \ref{sufficient} to the question when a cocycle $\A$ is 
cohomologous to a constant cocycle. Clearly, for a constant cocycle $\B$
 the standard holonomy is trivial, $H_{x,y}^{\B}=\Id$. Thus $\H^{\B,P}=\Id$
  for every $su$-cycle $P$ and hence  (i) becomes $\H^{\A,P}_{x_0}=\Id$.
Condition (ii) can be rewritten as $\B_{x_0}=C_{x_0}^{-1}  \circ (\H^{\A,P}_{{x_0},f{x_0}} )^{-1}\circ \A_{x_0} \circ C_{x_0}$ and so it defines a constant 
cocycle $\B$ uniquely for any choice of  $C_{x_0}$.  Thus we obtain
the first part of the following corollary. It was established in \cite{KN} 
for systems with local accessibility and for the standard holonomy of a 
cocycle satisfying a certain bunching assumption.

\begin{corollary}  \label{cohomologous to constant}
If a  cocycle $\A$ has a holonomy $H^{\A}$ satisfying 
\begin{equation} \label{trivial H}
\H^{\A,P}_{x_0}=\Id \quad  \text{ for 
very $su$-cycle $P_{x_0}$ based at some  point ${x_0}\in \M$},
\end{equation}   
then there exists a continuous conjugacy  between $\A$ and a constant
cocycle $\B$ that intertwines $H^{\A}$ and the standard holonomy 
$H^{\B}=\Id\,$ for $\B$.
Existence of such a holonomy $H^{\A}$ is a necessary condition
for $\A$  to be cohomologous to a constant cocycle.
\end{corollary}

The second part of the corollary follows from 
Proposition \ref{prop of inter conj} 
and the following observation: for any holonomy $\H^{\B}_{x,y}$ and any continuous conjugacy $C$ between $\A$ and $\B$, the formula
 $C(y) \circ  \H^{\B}_{x,y} \circ C(x)^{-1}$ defines a holonomy for $\A$. 
We note, however, that having the standard holonomy satisfy \eqref{trivial H}
is not a necessary condition for existence of a continuous 
conjugacy to a constant cocycle. Indeed, the cocycle $\A$ in 
Proposition \ref{nonreg example} is cohomologous 
to the constant cocycle $\B$ via a continuous $C$, but no conjugacy intertwines their standard holonomies. This together with Corollary \ref{cohomologous to constant}
implies that \eqref{trivial H} does not hold for the standard holonomy of $\A$. Also, the standard holonomy of $\A$ is mapped by $C$ 
to a non-standard holonomy for $\B$ for which \eqref{trivial H} does not hold.
In particular, holonomes for $\A$ and $\B$ are non-unique.


\section{Proofs}

\subsection{Proof of Theorem \ref{alpha Holder}.}

Since the cocycle $\A$  is fiber bunched and since by \eqref{partial def} 
$0<\nu(x)<\gamma(x)$, we can fix $\theta <1$ sufficiently close to 1 so that for all $x\in \M$, 
\begin{equation}\label{theta def}
\| \A_x\|\cdot \| \A_x^{-1}\| \cdot  \nu(x)^\beta \le \theta,
\quad  \| \A_x\|\cdot \| \A_x^{-1}\| \cdot  \hat \nu(x)^\beta \le \theta,
\;\text{ and }\;(\nu(x)/\gamma(x))^\beta< \theta.
\end{equation}
Since $\mh <1$ from \eqref{kappa} we can choose 
$\alpha$, $0<\alpha \le \beta$, sufficiently close to 0 so that 
\begin{equation}\label{alpha def} 
 \theta < (\mh (x) \nu (x) )^\alpha \quad \text{for all } x\in \M.
\end{equation}

By iterating points $x,y,x'',y''$ forward and using invariance of the holonomies (H3), we can assume without loss of generality that $\,y\in W^s_{\delta_0}(x)$, $\,y''\in W^s_{\delta_0}(x'')$ for some sufficiently small  $\delta_0>0$.
We denote $E^{cu}=E^{c}\oplus E^{u}$ and let $\Sigma_x$ be the exponential 
of the ball of radius $C_1\delta$ centered at $x$ in $E^{cu}(x)$. 
Since $E^{cu}$ is transversal to $E^s$, 
we can fix $C_1>0$ such that if $\delta$ is sufficiently small then for
 any $x\in \M$, $\Sigma_x$ is a submanifold transversal to $W^s$ and for any 
$x''$ with $\dist(x,x'')<\delta$ there is a unique intersection point $x'= \Sigma_x \cap W^s(x'')$. If $\delta$ is sufficiently small then the distances $\dist(x,x')$ 
and $\dist(x',x'')$ are at most $C_2\dist(x,x'')$, for some constant $C_2>0$ independent of
points $x,y,x'',y''$,
and also for each $z \in \Sigma_x$ the tangent space $T_z \Sigma_x$ is close 
to $E^{cu}(z)$. Similarly, we define $\Sigma_y$ and $y'$. 
By taking $\delta<\delta_0$ sufficiently small we can also ensure that $x',y,y' \in B_{2\delta_0}(x)$ and $\,y'\in W^s_{2\delta_0}(x')$.

First we iterate the points $x, x',y,y'$ and  estimate the distances between 
their trajectories in the next lemma. The setting and arguments here are similar to ones
in a direct proof of H\"older continuity of stable holonomies for a 
partially hyperbolic system, cf. \cite[Proposition 5.2]{W}.
We denote $x_k=f^kx$, and 
\begin{equation}\label{def nu_n}
\nu_k(x)=\nu(f^{k-1}x) \cdots \nu(fx)\, \nu(x)=\nu(x_{k-1})\cdots \nu(x_1) \, \nu(x_0).
\end{equation}
We will use similar notations for $x'$, $y$, and $y'$ as well as for the functions 
$\nh$, $\mh$, and $\g$.
We choose $n$ so that $\dist(x,x') \approx  \nu _n(x) \,\mh _n (x)$. 
More precisely, we take $n$ to be the largest integer satisfying the first
inequality in
\begin{equation}\label{n def}
 \,\dist(x,x')   \le \nu _n(x) \,\mh _n (x) \le C' \dist(x,x')  
\end{equation}  
This implies the second inequality with some constant  $C'$ independent of $x,x'$. 

\begin{lemma} \label{points}
Let $n$ be chosen according to \eqref{n def}. Then there exists $M$ such that
\begin{itemize}

\item[(a)]  $\dist_{W^s} (x_n, y_n)\le M \nu_n(x)\,$ and 
$\;\dist_{W^s} (x_n', y_n')\le M \nu_n(x);$ 
 \vskip.1cm
 
\item[(b)] $\dist(x_k,x_k') \le  \nu _n(x)  \,\mh _{n-k}(x_k)\,$ and 
 $\;\dist (y_k, y_k') \le 
  M \nu_n(x) \g_{n-k}(x_k)^{-1}$  \\for 
$\,0\le k\le n.$
\end{itemize}

\end{lemma}

\begin{proof} 
By continuity of the functions $\nu$, $\mh$, and $\g$ from \eqref{partial def} 
and \eqref{kappa}, there exists $0<r<1$ such for any point $p\in \M$ the
value at $p$ gives the corresponding estimate for any $q\in B_{r}(x)$.
It will be clear from the estimates that by taking $\delta_0$ and $\delta$ 
small enough, which forces $n$ to be large enough, we can ensure that
$x'_k,y_k,y'_k \in B_{r}(x_k)$ for each $0\le k\le n$. The first part of (b) 
follows since $\mh^{-1}$ bounds above the maximal expansion of $f$:
$$
   \dist(x_k,x_k') \le  \dist(x, x') \,\mh _k(x)^{-1} \le 
   \nu _n(x) \,\hat \mu_n (x) \,\mh _k(x)^{-1}= \nu _n(x)  \,\mh _{n-k}(x_k).
$$
Since $\,y\in W^s_{\delta_0}(x)$ and $\,y'\in W^s_{2\delta_0}(x')$ we obtain
$$ 
\dist_{W^s} (x_n, y_n)\le \delta_0 \,\nu_n(x) \quad\text{and}\quad
\dist_{W^s} (x_n', y_n')\le 2  \delta_0 \,\nu_n(x).
$$
Choosing $M= 3\delta_0 +1$ we obtain part  (a) and the estimate
$$ 
 \dist (y_n, y_n') \le \dist (x_n, x_n') +\dist (x_n, y_n) +
 \dist (x_n', y_n') \le  M\, \nu_n(x).
 $$
Since $\g$ is less than the strongest contraction along $E^{cu}$, we obtain 
the second part of (b):  
 $$
  \dist (y_k, y_k') \le   \dist (y_n, y_n') \gamma_{n-k}(x_k)^{-1}\le
  M \, \nu_n(x) \gamma_{n-k}(x_k)^{-1}
 $$
  for $k=0, 1, \dots , n$.
For this we note that the transversals $\Sigma_x$ and 
 $\Sigma_y$ are chosen close to $E^{cu}$ and that their forward iterates
 $f^k(\Sigma_x)$ and $f^k(\Sigma_y)$ will remain close to $E^{cu}$.
\end{proof}
\vskip.2cm

Now we estimate the holonomies. For simplicity, in this proof we use $H$ 
for the standard stable holonomy $H^{\A,s}$. Our goal is to show that 
\begin{equation}   \label{H'}
\|H_{x',y'} \circ H_{x,y}^{-1} -\Id \| \le C_{14} \dist(x,x')^\alpha.
 \end{equation}
Note that all relevant holonomies between points $x, x',x'',y,y',y''$ lie in
a compact subset of $GL(V)$. Thus, once \eqref{H'} is established, 
Lemma \ref{eqdist} implies a H\"older estimate  for $d(H_{x,y}, H_{x',y'})$ 
similar to \eqref{H'}. Also, since $H_{x'',y'} ^{-1} \circ H_{x',y'} = H_{x',x''}$, 
 (H4) and the estimate $\dist(x',x'')\le C_2\dist(x,x'')$
give a $\beta$-H\"older estimate for $d(H_{x',y'}, H_{x'',y'})$.
Similarly, $\dist(y',y'')\le C_2\dist(y,y'')$ gives a $\beta$-H\"older estimate
for  $d(H_{x'',y'}, H_{x'',y''})$.
We conclude that \eqref{H'} yields the desired $\alpha$-H\"older estimate
for $d(H_{x,y}, H_{x'',y''})$ and proves the theorem. To prove \eqref{H'} we write
\begin{equation} \label{H H inv}
\begin{aligned}
H_{x',y'} \circ H_{x,y}^{-1} & \,= 
\left( (\A^n_{y'})^{-1} \circ H_{x_n',y_n'} \circ \A^n_{x'} \right) \circ
\left( (\A^n_{y})^{-1} \circ H_{x_n,y_n} \circ \A^n_{x} \right)^{-1}= \\
& \,=(\A^n_{y'})^{-1} \circ H_{x_n',y_n'} \circ 
\left( \A^n_{x'}  \circ (\A^n_{x})^{-1} \right)
\circ (H_{x_n,y_n})^{-1}   \circ \A^n_{y}= \\
&\, =(\A^n_{y'})^{-1} \circ ( \Id + \Delta_1) \circ 
( \Id + \Delta_2)\circ ( \Id + \Delta_3)   \circ \A^n_{y},
\end{aligned}
\end{equation} 
where
$$
\Delta_1 = H_{x_n',y_n'}-\Id, \qquad 
\Delta_2 = \A^n_{x'}  \circ (\A^n_{x})^{-1} - \Id, \qquad 
\Delta_3 = H_{y_n,x_n}-\Id.
$$
By (H4) and Lemma \ref{points}(a) we have
$$
\|\Delta_1\|=\|H_{x_n',y_n'}-\Id\|\le K \dist_{W^s} (x_n',y_n')^\beta\le
KM^\beta \nu_n(x)^\beta,
$$
and similarly $\|\Delta_3\|\le KM^\beta \nu_n(x)^\beta.\;$ 
Also, by Lemma \ref{Anxx'} below we have
$$
\|\Delta_2\|=\| \A^n_{x'}  \circ (\A^n_{x})^{-1} - \Id\| \le C_7 \nu_n(x)^\beta.
$$
Therefore, from \eqref{H H inv} we obtain
\begin{equation}\label{H estimate}
\|H_{x',y'} \circ H_{x,y}^{-1} -\Id \| \le 
\|(\A^n_{y'})^{-1}  \circ \A^n_{y} -\Id \| + 
\|(\A^n_{y'})^{-1} \| \cdot \|\A^n_{y}  \| \cdot C_{12}\nu _n(x)^\beta.
\end{equation}
Equation \eqref{H estimate} and Lemma \ref{Anyy'} imply that
$$
\|H_{x',y'} \circ H_{x,y}^{-1} -\Id \| \le C_{11} \theta ^n + C_9 \,\theta^n \nu_n (x)^{-\beta}  C_{12} \, \nu _n(x)^\beta \le C_{13} \theta ^n,
$$
and by the choices of $\alpha$  and $n$, \eqref{alpha def} and \eqref{n def}, 
we conclude that
$$
\|H_{x',y'} \circ H_{x,y}^{-1} -\Id \| \le C_{13}\theta ^n 
 \le  C_{13} \, (\mh _n (x) \nu_n (x))^{\alpha} \le C_{13} \, (C' \dist(x,x'))^\alpha.
$$
This completes the proof of the theorem  modulo Lemmas \ref{points}, 
\ref{Anxx'}, and \ref{Anyy'}.

\vskip.3cm

\begin{lemma} \label{Anxx'}
$\;\|\A^n_{x'}\circ (\A^n_x)^{-1} -\Id\,\|  \,\le\, C_7  \nu_n(x)^\beta.$
\end{lemma}

\begin{proof} 

 We rewrite $\,\A^n_{x'}\circ (\A^n_x)^{-1} \;$ as follows
\begin{equation}
 \begin{aligned} \label{ArA}
 & \A^n_{x'}\circ (\A^n_x)^{-1} \,= 
\A^{n-1}_{x'_1} \circ  \A_{x_0'} \circ (\A_{x_0})^{-1} \circ ( \A^{n-1}_{x_1})^{-1} =\\
  & = \A^{n-1}_{x'_1}  \circ (\Id+r_0) \circ  ( \A^{n-1}_{x_1})^{-1} =\\
  & = \A^{n-1}_{x'_1}  \circ   ( \A^{n-1}_{x_1})^{-1}  +
 \A^{n-1}_{x'_1}  \circ r_0 \circ  ( \A^{n-1}_{x_1})^{-1}   = \dots =\\
 &=\Id+\sum_{i=1}^{n} \A_{x'_i}^{n-i}\circ r_{i-1}\circ (\A^{n-i}_{x_i})^{-1} ,
 \quad \text{where  }\, r_{i}=\Id -(\A_{{x}_i'})^{-1} \circ \A_{x_i}. 
 \end{aligned}
\end{equation}
First we estimate $\| r_{i} \|$ using boundedness of $ \| (\A_{x_i'})^{-1} \|$ 
and Lemma \ref{points}  (b):
\begin{equation} \label{r_i}
\begin{aligned}
  \|r_{i}\| & = \,\| \Id - (\A_{x_i'})^{-1} \circ \A_{x_i} \| \, \le  \,
 \| (\A_{x_i'})^{-1} \| \cdot \| \A_{x_i'} -  \A_{x_i} \| \\
 & \le  C_3 \cdot\dist(x_{i}, x'_{i})^\beta 
    \le C_3 \left (  \nu _n(x)  \, \hat \mu _{n-i}(x_i) \right)^\beta. 
\end{aligned}
\end{equation}
\vskip.2cm

\noindent 
Next we estimate $\| \A^{n-i}_{x_i'} \| \cdot  \| (\A^{n-i}_{x_i} )^{-1} \| $.
Using   H\"older continuity of  $\A$  we obtain
$$
\begin{aligned}
  \frac{\|\A_{x_k'}\|}{\|\A_{x_k}\|} =   \frac{\|\A_{x_k}+\A_{x_k'} -\A_{x_k}\|}{\|\A_{x_k}\|} \;\le\;
    1+ \frac{\|\A_{x_k'}-\A_{x_k}\| }{\|\A_{x_k}\|} \;\le\; 
  1+ C_4\,\dist (x_k,x_k')^\beta.
  \end{aligned}
 $$ 
Hence we obtain using \eqref{theta def} that
$$
\begin{aligned}
&\| \A^{n-i}_{x_i'} \| \, \| (\A^{n-i}_{x_i} )^{-1} \|  \le 
  \prod_{k=i}^{n-1} \|\A_{x_k'} \| \cdot  \prod_{k=i}^{n-1} \|(\A_{x_k})^{-1} \| \le
\prod_{k=i}^{n-1} \|\A_{x_k}\| \, \|(\A_{x_k})^{-1}\|  \cdot 
   \prod_{k=i}^{n-1} \frac{\|\A_{x_k'}\|}{\|\A_{x_k}\|} \\ 
 & \le \;  \prod_{k=i}^{n-1}  \theta \,\hat \nu (x_k)^{-\beta} \cdot  \prod_{k=i}^{n-1} 
      \left( 1 + C_4 \left( \dist (x_k, x_k') \right)^\b \right) \le \\
&\le \; \theta^{n-i} \hat \nu _{n-i}(x_i) ^{-\beta} \cdot  \prod_{k=i}^{n-1} 
      \left( 1 + C_4 \left( \nu _n(x)  \hat \mu _{n-k}(x_k) \right)^\b \right)
      \le C_5 \,\theta^{n-i} \hat \nu _{n-i}(x_i) ^{-\beta},
      \end{aligned}
 $$ 
as the product is uniformly bounded in $n$ and $i$ since $\nu ,\mh <1$. In particular,
\begin{equation}   \label{F_i'}
\| \A^{n}_{x'} \| \cdot  \| (\A^{n}_{x} )^{-1} \|  
 \le C_5 \,\theta^{n} \,\hat \nu _{n}(x) ^{-\beta}. 
 \end{equation}
Now using \eqref{ArA} and  \eqref{r_i} we conclude that
$$
\begin{aligned}
&  \|\A^n_{x'}\circ (\A^n_x)^{-1} -\Id \,\|  \,\le\, \sum_{i=1}^{n} 
  \| \A^{n-i}_{x'_i}\circ r_{i-1}\circ( \A^{n-i}_{x_i})^{-1} \| \le  \\
&\le \, \sum_{i=1}^{n} C_3 \, \nu _n(x)^\beta   \hat \mu _{n-i+1}(x_{i-1}) ^\beta
\cdot C_5 \, \theta^{n-i} \hat \nu _{n-i}(x_i) ^{-\beta} \le \\
&\le \,C_6 \, \nu _n(x)^\beta \, \sum_{i=1}^{n} 
 \theta^{n-i} \, ( \hat \nu _{n-i}(x_i)^{-1}\hat \mu _{n-i}(x_i) ) ^{\beta}
 \le C_7 \, \nu _n(x)^\beta
\end{aligned}
  $$
since $\theta<1$ and $\mh < \nh$.  This completes the proof of Lemma \ref{Anxx'}.
 \end{proof}

\begin{lemma} \label{F_iS}
There exists $C_8$ such that if $w \in W^s_{r}(z)$ then for any $k \in \N$
 $$ \;  \|\, \A^k_w\,\|  \le C_8 \|\, \A^k_z\,\|\qquad \text {and} \qquad
\| ( \A^k_w )^{-1} \| \le  C_8 \| ( \A^k_z )^{-1} \|. 
$$  
\end{lemma}
\begin{proof}
From (H3) we have that 
$\A^k_w = H_{f^k z,f^k w} \circ  \A^k_z \circ H_{z,w}^{-1}$
and obtain the first inequality since the norms of $H_{z,w}$ 
and $H_{z,w}^{-1}$ are bounded uniformly in $z \in \M$ and 
$w \in W^s_{r}(z)$ by compactness. The second one is established similarly.
\end{proof}

\begin{lemma} \label{Anyy'}
$\| \A^{n}_{x'} \| \cdot  \| (\A^{n}_{x} )^{-1} \|  \le 
C_9 \theta^{n}  \nu _{n}(x) ^{-\beta}$
and $\| (\A^n_{y'})^{-1}\circ \A^n_y - \Id \| \,\le \,  C_{11} \theta ^n$.

\end{lemma}
\begin{proof} 
First we claim that $\| \A^{i}_{x'} \| \cdot  \| (\A^{i}_{x} )^{-1} \|  
 \le C_5 \theta^{i}  \nu _{i}(x) ^{-\beta}$ for $0 \le i \le n$. 
This is obtained in the same way as \eqref{F_i'} using the first inequality
in \eqref{theta def} instead of the second one.
Applying the previous lemma we also obtain
$$ \;\| ( \A^i_{y'})^{-1} \| \le  C_8 \| ( \A^k_{x'} )^{-1} \| \qquad \text {and} \qquad
 \| \A^i_y\|  \le C_8 \| \A^i_x \|
$$
for all $i \in \N$.
We conclude that for each $0 \le i\le n$, 
$$ \;\| ( \A^i_{y'} )^{-1} \| \cdot  \|\, \A^i_y\,\| \le C_9 \,\theta^i \, \nu_i (x)^{-\beta}$$
giving, in particular, the first inequality in the lemma.

Similarly to  \eqref{ArA} and \eqref{r_i} we obtain using Lemma \ref{points} (b) that
$$
\begin{aligned}
&  (\A^n_{y'})^{-1}\circ \A^n_y \, 
 =\Id+\sum_{i=0}^{n-1} (\A^{i}_{y'})^{-1}\circ r_{i}\circ \A^i_y ,
 \quad \text{where  }r_i=\Id - (\A_{y_i'})^{-1} \circ \A_{y_i} \\
&  \text{satisfy  } \quad \| r_i \| = \|(\A_{y_i'})^{-1} \circ \A_{y_i} -\Id\| \le C_3\,\dist (y_i, y_i ')^\beta
\le  M C_3 \nu_n(x)^\beta \gamma_{n-i}(x_i)^{-\beta}.
\end{aligned}
$$
Using that $\nu(x)^\beta (\gamma(x)^\beta \theta)^{-1}<1$ by \eqref{theta def},
we conclude that
$$
\begin{aligned}
&\| (\A^n_{y'})^{-1}\circ \A^n_y - \Id \| \,\le \,
\sum_{i=0}^{n-1} \|(\A^{i}_{y'})^{-1}\|\cdot \|\A^i_y\| \cdot \| r_{i}\| \,\le \\
&\le M C_3 C_9\sum_{i=0}^{n-1}\theta^{i}  \nu _{i}(x) ^{-\beta} \nu_n(x)^\beta \gamma_{n-i}(x_i)^{-\beta}
\le C_{10} \theta^{n}\sum_{i=0}^{n-1}\theta^{i-n}  \nu _{n-i}(x_i) ^{\beta} \gamma_{n-i}(x_i)^{-\beta} \le C_{11} \theta ^n.
\end{aligned}
$$
\end{proof}

\subsection{Proof of Proposition \ref{constant}}
Let $\rho_1 < \dots <\rho_l$ be the distinct moduli of the eigenvalues of 
the matrix $A$. Let $\R^d = E_1 \oplus  \dots \oplus E_l$  be the 
corresponding splitting into direct sums of the 
generalized eigenspaces, and let $\A_i=\A| _{E_i}$. Then
for any $\e >0$ there exists $K_\e$ such that 
$$
K_\e^{-1} (\rho_i-\e)^n \le \| \A_i^n v \| \le K_\e (\rho_i+\e)^n
\quad \text{for any unit vector }v\in E_i.
$$ 
Then any sufficiently $C^0$ small H\"older continuous perturbation 
$\B$ of $\A$ has a H\"older continuous invariant splitting with similar 
estimates for the corresponding restrictions $\B_i$ (cf. \cite[Theorems 3.4 and 3.8]{Pe}). It follows that $\B_i$'s 
are close to conformal and satisfy the weaker fiber bunching condition
\eqref{weak fiber bunched}. 
Hence by Remark \ref{weaker} $\B_i$'s have standard holonomies, which 
combine into the standard holonomy for $\B$. We note, however, 
that the H\"older exponent of the splitting and of the resulting holonomy  
may be lower than that  of $\B$.

\subsection{Proof of Theorem \ref{measurable cohomology}.}

Let $H^{\A}$ be the standard holonomies for $\A$, which exist by the assumption. Since $\B$ is  uniformly quasiconformal, it satisfies the 
weaker fiber bunching condition \eqref{weak fiber bunched}.
 Thus, by Remark \ref{weaker}, $\B$ has standard holonomies, 
 which we denote by $H^{\B}$.

Our main goal is to show that $C$ intertwines the holonomies of $\A$ and 
$\B$ on a set of full measure. More precisely, for the stable holonomies we 
will show that there exists a subset $Y$ of $\M$ with $\m(Y)=1$ such that 
\eqref{intertwines} holds for all $x,y \in Y$ such that $y\in W^s(x)$. A similar 
statement holds for the unstable holonomies.

By the assumption, there is a set of full measure $Y_1\subset \M$ such that 
 for all $x\in Y_1$, $\,\A_x=C(fx)\circ \B_x \circ C(x)^{-1}$.   Since the function 
 $C$ is $\m$-measurable and $GL(V)$ is separable, by Lusin's theorem there 
exists a compact set $S\subset \M$ with $\m(S)>1/2$ such that 
$C$ is uniformly continuous on  $S$. 
It follows that $\| C\|$ and $\|C^{-1} \|$ are bounded on $S$.
Let $Y_2$ be the set of points in $\M$
for which the frequency of visiting $S$ equals $\m(S)$.
By Birkhoff ergodic theorem, $\m(Y_2)=1$. 
\vskip.1cm
Let $Y=Y_1\cap Y_2$.
Clearly, $\m(Y)=1$ and we can assume
 that the sets $Y_1, Y_2, Y$ are $f$-invariant.
Suppose that $x,y\in Y$ and 
  $y \in W^s_{R}(x)$ for some fixed radius $R$. Then 
\begin{equation} \label{meascoh1}
\begin{aligned}
&(\A^n_y)^{-1}\circ \A^n_x= 
\left( C(f^ny)\circ \B^n_y\circ C(y)^{-1}\right)^{-1} \circ 
C(f^nx)\circ  \B^n_x\circ  C(x)^{-1}= \\
&= C(y) \circ (\B^n_y)^{-1}\circ  C(f^n y)^{-1} \circ 
 C(f^nx)\circ \B^n_x \circ C(x)^{-1}= \\
&= C(y) \circ (\B^n_y)^{-1}  \circ (\Id +\Delta_n) \circ \B^n_v \circ C(x)^{-1}= \\
&= C(y) \circ (\B^n_y)^{-1} \circ \B^n_x\circ C(x)^{-1} + 
C(y)\circ (\B^n_y)^{-1} \circ \Delta_n \circ \B^n_x \circ C(x)^{-1} . 
\end{aligned}
\end{equation}
We will show that the second term in the last line tends to 0 along a subsequence.
First we estimate the norm of $\Delta_n$.
\begin{equation} \label{Delta_n}
  \| \Delta_n \|= \| C(f^n y)^{-1} \circ C(f^n x) - \Id\| \le
   \| C(f^n y)^{-1} \| \cdot \|C(f^n x) - C(f^n y)\|.
\end{equation}
Since $x,y \in Y_2\subset Y$, there exists a sequence 
$\{n_i\}$ such that $f^{n_i}x, f^{n_i}y \in S$ for all $i$.
Since  $y \in W^s_{R}(x)$, $\,\dist (f^{n_i}x, f^{n_i}y)\to 0$
 and hence $\|C(f^{n_i} x) - C(f^{n_i} y)\| \to 0$ by uniform continuity 
of $C$ on $S$. As $\|C^{-1}\|$ is uniformly bounded on $Y$, 
\eqref{Delta_n} implies 
$$
\| \Delta_{n_i} \| \to 0 \quad\text{ as }i\to\infty
$$
\vskip.1cm

Using Lemma \ref{F_iS} and quasiconformality of $\B$ we also obtain that
\begin{equation}\label{B}
\| (\B^n_y)^{-1}\| \cdot \| \B^n_x \|
\le  \| (\B^n_y)^{-1}\| \cdot C_8 \| \B^n_y \|
\le  C_8 \, K(\B)
\end{equation}
for all $x\in \M$ and $y\in W^s_{R}(x)$. Now it follows that 
$$
\| C(y)\circ (\B^{n_i}_y)^{-1} \circ \Delta_{n_i} \circ \B^{n_i}_x \circ C(x)^{-1} \|
\to 0 \quad\text{ as }i\to\infty.
$$
Since the holonomies $H^{\A,s}$ and $H^{\B,s}$ are standard, i.e. satisfy \eqref{as limits},  passing to the limit in 
\eqref{meascoh1}  along the sequence $n_i$ yields
\begin{equation}\label{full measure}
H_{x,y}^{\A,s}=C(y)\circ H_{x,y}^{\B,s} \circ C(x)^{-1}
\quad\text{for all } x,y\in Y \text{ such that }y\in W^s_{R}(x).
\end{equation}
We conclude that $C$ intertwines the holonomies $H^{\A}$ and $H^{\B}$ 
on a set of full measure.

It follows that $C(y) =H_{x,y}^{\A,s}\circ C(x)\circ (H_{x,y}^{\B,s})^{-1} $ and, 
by continuity of holonomies, we conclude that $C$ is so called {\em essentially s-continuous}
in the sense of \cite{ASV}. 
Similarly, $C$ is {\em essentially u-continuous}. 
By the assumption on the base system ($f$ is center bunched and accessible),
\cite[Theorem D]{ASV} implies that $C$ coincides on a set of full measure 
with a continuous function $\tilde C$. It follows that $\tilde C$ is a conjugacy
between $\A$ and $\B$ and, by \eqref{full measure}, intertwines $H^{\A}$ and $H^{\B}$.

\subsection{Proof of Proposition \ref{cohomology 2}.}

 As in the proof of Theorem  \ref{measurable cohomology}
 we obtain \eqref{meascoh1}.
 Since $C$ is $\beta$-H\"older, for any $x\in \M$ and  $y \in W^s_{R}(x)$
 $$
 \begin{aligned}
 \|\Delta_n\|& \le    \| C(f^n y)^{-1} \| \cdot \|C(f^n x) - C(f^n y)\|   \\
 & \le K_1\dist(f^nx, f^n y)^\beta\le K_2 \,\nu_n(x)^\beta \dist (x,y)^\beta.
 \end{aligned}
$$ 
Using fiber bunching of $\B$ we choose $\theta<1$ as in \eqref{theta def},
and by Lemma \ref{F_iS} we obtain
$$
\begin{aligned}
&\|(\B^n_y)^{-1} \circ \Delta_n \circ \B^n_x\| \le 
\| (\B^n_y)^{-1}\| \cdot \| \B^n_x \|  \cdot \|\Delta_n\| \\
&\le  C_8 \| (\B^n_x)^{-1}\| \cdot \| \B^n_x \|   \cdot  K_2 \,\nu_n(x)^\beta \dist (x,y)^\beta \le K_3 \,\theta^n \, \dist(x,y)^\beta .
\end{aligned}
$$
It follows that the second 
term in last line of \eqref{meascoh1} tends to 0 as $n\to\infty$ for every $x\in \M$ 
and $y\in W^s_{R}(x)$. Passing to the limit in \eqref{meascoh1} we conclude that
 $C$ intertwines the standard holonomies of $\A$ and $\B$.

\subsection{Proof of Proposition \ref{nonreg example}.}
We use the construction described in \cite[Theorem 5.5.3]{KN} which was
based on an example by R. de la Llave \cite{L}.
Let $f$ be an Anosov automorphism of $\T^2$ with eigenvalues 
$\lambda >1$ and $\lambda^{-1}$.  We fix a number $r$, where $\beta'<r<\beta$,
and set $\mu = \lambda ^r$.
We consider smooth $GL(2,\R)$-valued  cocycles over  $f$
 $$
\B= \left[ \begin{array}{cc} \mu & 0 \\ 0 & 1 \end{array} \right]
\quad\text{and} \quad
\A(x)=\left[ \begin{array}{cc} \mu& \phi(x) \\ 0 & 1 \end{array} \right] 
$$
Then the constant cocycle $\B$ is $\beta$ fiber bunched.
We take $\phi$ sufficiently small so that $\A$ is sufficiently $C^0$ close to $\B$
 and hence it is also $\beta$ fiber bunched. 
Hence both $\A$ 
 and $\B$ have standard stable and unstable holonomies which 
 are $\beta$-H\"older along the leaves of the corresponding foliation, 
 i.e. satisfy (H4).
 
 We take $\e$ such that $\beta '<r-\epsilon$ and  ${r+\epsilon}<\beta$.
By Theorem 5.5.3 in \cite{KN},  there exist 
arbitrarily $C^\infty$ small functions $\phi(x)$ such that $\A$ and $\B$ 
 are cohomologous via a $C^{r-\epsilon}$ conjugacy, but not 
via a $C^{r+\epsilon}$ conjugacy. 
Thus there is a $\beta'$-H\"older conjugacy $C$ between  $\A$ and $\B$, 
but no $\beta$-H\"older conjugacy.
It follows that no conjugacy $\tilde C$ can intertwine the standard 
holonomies of $\A$ and $\B$.  Indeed, 
otherwise $\tilde C$ would  be $\beta$-H\"older along the stable and unstable
leaves of $f$, since so are the standard holonomies, and hence it
would be $\beta$-H\"older on $\T^2$.
 
 In this example, the low regularity of $C$ is due to the low regularity
 of the unique invariant expanding sub-bundle $\mathcal V$ for $\A$, 
which has to be mapped by $C$ to the first coordinate line. In fact,
$C$ and $\mathcal V$ are smooth along  the stable leaves of $f$, and
$C$ intertwines the standard stable holonomies of $\A$ and $\B$,
but not the unstable ones.

\subsection{Proof of Theorem \ref{sufficient}.}
In the proof we will use $x$ in place of $x_0$ to simplify notations.
We define $C(x)=C_x$, and then for every $y\in \M$ we define 
$$
C(y)=\H^{\A,P}_{x,y} \circ C(x)\circ  (\H^{\B,P}_{x,y})^{-1},
$$
where $P_{x,y}$ is an $su$-path from $x$ to $y$. Note that 
$C \mapsto \H^{\A,P}_{x,y} \circ C \circ  (\H^{\B,P}_{x,y})^{-1}$
defines a map from the group $GL(V)$ of operators on the fiber at $x$ to
the one on the fiber at $y$, and that a concatenation of paths corresponds
to the composition of the maps. Therefore, it is easy to check that the  
assumption (i) implies that $C(y)$ is independent of the $su$-path $P$ 
and hence is well-defined.
In particular, it follows that for any $y,z\in \M$ and  any $su$-path 
$P_{y,z}$ from $y$ to $z$
\begin{equation} \label{sufficient1}
C(z)=\H^{\A,P}_{y,z} \circ C(y)\circ  (\H^{\B,P}_{y,z})^{-1}.
\end{equation} 
Hence continuity of holonomies implies that the function $C$ is continuous along 
the stable and unstable foliations of $f$.
Since $f$ is accessible, this implies continuity of $C$  on $\M$ by \cite[Theorem E]{ASV}.

It remains to show that $C$ satisfies the cohomological equation.
Consider any $y\in \M$ and fix an $su$-path $P=P_{x,y}$ from $x$ to $y$.
Then $fP$ is an  $su$-path from $fx$ to $fy$.  
By property (H3) of holonomies we obtain using \eqref{sufficient1}
with $z=fy$ and $y=fx$ that
$$
\begin{aligned}
C(fy) &= \H^{\A,fP}_{fx,fy} \circ C(fx)\circ  (\H^{\B,fP}_{fx,fy})^{-1}=  \\
&=\A_y \circ \H^{\A,P}_{x,y} \circ \A_x^{-1} \circ C(fx)\circ
\B_x \circ (\H^{\B,P}_{x,y})^{-1} \circ \B_y^{-1}.
\end{aligned}
$$
By assumption (ii) and  \eqref{sufficient1},
$$
C(fy)=\A_y \circ \H^{\A,P}_{x,y}\circ C(x) \circ (\H^{\B,P}_{x,y})^{-1} 
\circ \B_y^{-1}
=\A_y \circ C(y)\circ  \B_y^{-1} ,
$$
and we conclude that $C$ is a  conjugacy.


\end{document}